\documentclass{article}
\usepackage{spconf,amsmath,graphicx}

\usepackage{pbox}

\usepackage{xcolor}

\usepackage{algorithm}
\usepackage{algpseudocode}
\usepackage{amsmath}
\usepackage{graphics}
\usepackage{epsfig}

\usepackage{amssymb,amsfonts}

\usepackage{amsthm}
\theoremstyle{plain}
\newtheorem{thm}{Theorem}[section]
\newtheorem{lem}[thm]{Lemma}
\newtheorem{prop}[thm]{Proposition}
\newtheorem{cor}[thm]{Corollary}

\theoremstyle{definition}

\theoremstyle{remark}

\newcommand{\xtrue}{x^\natural}
\newcommand{\norm}[1]{\left\Vert #1 \right\Vert}
\newcommand{\set}[1]{\left\{ #1 \right\}}
\newcommand{\abs}[1]{\left\vert #1 \right\vert}

\newcommand{\lmin}{\underline{\mu}}
\newcommand{\lmax}{\bar{\mu}}

\DeclareMathOperator*{\argmin}{arg\,min}

\newcommand{\iprod}[2]{\langle #1 , #2 \rangle}
\newcommand{\naf}{\nabla g}

\newcommand{\dA}{d_{\Phi}}
\newcommand{\dB}{d_{\Psi}}
\newcommand{\dX}{d_{\mathcal{C}}}
\newcommand{\dy}{d_{\eta}}

\title{Frank-Wolfe Works for Non-Lipschitz Continuous Gradient Objectives: Scalable Poisson Phase Retrieval}
\name{
\pbox{20cm}{Gergely~Odor\textsuperscript{*}, Yen-Huan~Li\textsuperscript{*}, Alp~Yurtsever\textsuperscript{*}, Ya-Ping~Hsieh\textsuperscript{*}, \\ Quoc~Tran-Dinh\textsuperscript{*\ \textdagger}, Marwa~El~Halabi\textsuperscript{*}, and Volkan~Cevher\textsuperscript{*}} \thanks{This work was supported in part by ERC Future Proof, SNF 200021-146750 and SNF CRSII2-147633.}}
\address{\textsuperscript{*}Laboratory for Information and Inference Systems\\
\'{E}cole Polytechnique F\'{e}d\'{e}rale de Lausanne, Switzerland \\[5pt]
\textsuperscript{\textdagger}Department of Statistics and Operations Research\\
The University of North Carolina at Chapel Hill, USA}

\begin{document}
\ninept
\maketitle
\begin{abstract}
We study a phase retrieval problem in the Poisson noise model. Motivated by the PhaseLift approach, we approximate the maximum-likelihood estimator by solving a convex program with a nuclear norm constraint. While the Frank-Wolfe algorithm, together with the Lanczos method, can efficiently deal with nuclear norm constraints, our objective function does not have a Lipschitz continuous gradient, and hence existing convergence guarantees for the Frank-Wolfe algorithm do not apply. In this paper, we show that the Frank-Wolfe algorithm works for the Poisson phase retrieval problem, and has a global convergence rate of $O ( 1/t )$, where $t$ is the iteration counter. We provide rigorous theoretical guarantee and illustrating numerical results.
\end{abstract}
\section{Introduction}

Phase retrieval is the problem of estimating a complex-valued signal from intensity measurements, which arises in many applications such as X-ray crystallography, diffraction imaging, astronomical imaging, and many others \cite{Shechtman2015}.

We focus on the Poisson noise case in this paper. Formally speaking, we are interested in estimating a signal $x^\natural \in \mathbb{C}^p$, given $a_1, \ldots, a_n \in \mathbb{C}^p$ and measurement outcomes $y_1, \ldots, y_n$, modeled as independent random variables following the Poisson distribution: 
\begin{equation}
\mathbb{P} \left\{ y_i = y \right\} = \frac{\exp \left( - \lambda_i \right) \lambda_i^y }{y!}, \quad y \in \{ 0 \} \cup \mathbb{N} \notag
\end{equation}
where $\lambda_i := \left\vert \left\langle a_i, x^\natural \right\rangle \right\vert^2$ for all $i$. In practice, each $y_i$ represents the number of photons detected by the sensor \cite{Fienup1982}.


The corresponding maximum-likelihood (ML) estimation yields a non-convex optimization problem which is difficult to solve.
A recent approach to circumvent this computational issue is PhaseLift \cite{Candes2015,Candes2013}. 
The PhaseLift approach casts the phase retrieval problem as a low rank matrix recovery problem, and then we can apply any convex optimization-based estimator, such as the basis pursuit like estimator \cite{Recht2010}, the nuclear-norm penalized estimator \cite{Candes2011b}, and the Lasso like estimator \cite{Davenport2014}.

Following the PhaseLift approach, we show in Section \ref{sec_formulation} that we can recover $\xtrue$ by solving
\begin{equation}
\hat{X} \in \argmin_{X} \left\{ f ( X ) : X \in \mathcal{X}  \right\} , \label{eq_opt}
\end{equation}
where
\begin{align}
f ( X ) &:=  \sum_{i = 1}^n \left\{ - y_i \log \left[ \mathrm{Tr} \left( A_i X \right) \right] + \mathrm{Tr} \left( A_i X \right) \right\}, \label{eq_f} \\
\mathcal{X} &:= \set{ X \geq 0, \ \norm{ X }_* \leq c, \ X \in \mathbb{C}^{p \times p} }. \label{eq_X}
\end{align}
for some $c > 0$, $A_i := a_i a_i^H$. A rule of thumb for choosing $c$ can be found in Section \ref{sec_constraint}.
We then find an eigenvector associated with the largest eigenvalue of $\hat{X}$ as our estimate of $x^\natural$.

It is easy to check that (\ref{eq_opt}) is a convex optimization problem. Existing convex optimization tools, however, are not directly applicable to solving (\ref{eq_opt}) due to two issues.

\begin{enumerate}
\item Most existing algorithms, such as \cite{Tran-Dinh2015b}, are computationally expensive for nuclear norm constraints, as they require computing the eigenvalue decomposition of a matrix in $\mathbb{C}^{p \times p}$ at each iteration.
\item While Frank-Wolfe-type algorithms can be relatively scalable for nuclear norm constraints \cite{Jaggi2013}, existing theoretical convergence guarantees for these Frank-Wolfe-type algorithms are not valid for our loss function in (\ref{eq_opt}).
\end{enumerate}

We will address the issues in detail in Section \ref{sec_issues}.

In this paper, we show that the standard Frank-Wolfe algorithm works for the optimization problem (\ref{eq_opt}), with a properly chosen parameter to be explicitly specified in Theorem \ref{thm_main}. Our theorem guarantees that the Frank-Wolfe algorithm converges at the rate $O ( 1 / t )$ globally, where $t$ is the iteration counter. 
Numerical experiments show that the empirical convergence rate can be even faster. 
The algorithm shares the same merit of the standard Frank-Wolfe algorithm, in the sense that it is scalable when dealing with a nuclear norm constraint. 

To the best of our knowledge, this is the first theoretical guarantee for the Frank-Wolfe algorithm applied to a non-H\"{o}lder (and hence non-Lipschitz) continuous gradient objective function.


\section{Poisson Phase Retrieval by Convex Optimization} \label{sec_formulation}

For the Poisson noise model, the ML estimator of $x^\natural$ is given by
\begin{equation}
\hat{x}_{\text{ML}} \in \arg \min_{x} \left\{ L ( x ) : x \in \mathbb{C}^p \right\} \label{eq_ml}
\end{equation}
where $L$ is the negative log-likelihood function (under a constant shift):
\begin{equation}
L ( x ) := \sum_{i = 1}^n \left[ - y_i \log \left( \left\vert \left\langle a_i, x \right\rangle \right\vert^2 \right) + \left\vert \left\langle a_i, x \right\rangle \right\vert^2 \right]. \notag
\end{equation}
The function $L$, unfortunately, is non-convex, and currently there does not exist a well-guaranteed algorithm for solving the optimization problem.

Motivated by the PhaseLift approach \cite{Candes2015,Candes2013}, we can reformulate the non-convex optimization problem (\ref{eq_ml}) as follows. Define $A_i := a_i a_i^H$ for all $i$, and $X^\natural := x^\natural ( x^\natural )^H$. Then we have 
\begin{align}
\left\vert \left\langle a_i, x^\natural \right\rangle \right\vert^2 &= \mathrm{Tr} \left( A_i X^\natural \right) \quad i = 1, \ldots, n \notag
\end{align}
where $\mathrm{Tr} \left( \cdot \right)$ denotes the trace function, and hence we can rewrite the original optimization problem as
\begin{equation}
\hat{x}_{\text{ML}} \in \arg \min_{x} \left\{ f ( X ) : X = x x^H, x \in \mathbb{C}^p \right\} \notag
\end{equation}
where $f$ is given in (\ref{eq_f}). This is equivalent to the optimization problem
\begin{equation}
\hat{X}_{\text{ML}} \in \arg \min_{X} \left\{ f ( X ) : X \geq 0, \mathrm{rank} ( X ) = 1, X \in \mathbb{C}^{p \times p} \right\}. \notag
\end{equation}
Note that given $\hat{X}_{\text{ML}}$, $\hat{x}_{\text{ML}}$ can be recovered via the relation $\hat{X}_{\text{ML}} = \hat{x}_{\text{ML}} \hat{x}_{\text{ML}}^H$.

As the variable $X$ is always of rank $1$, we can consider the convex relaxation given in (\ref{eq_opt}). We then find an eigenvector associated with the largest eigenvalue of $\hat{X}$ as our estimate of $x^\natural$.

It is easy to verify that (\ref{eq_opt}) is a convex optimization problem.

\section{A Rule of Thumb for Setting the Constraint} \label{sec_constraint}
In the convex optimization formulation (\ref{eq_opt}), we leave one parameter $c$ unspecified. The ideal setting should be $c = \norm{ X^\natural }_* = \norm{ \xtrue }_2^2$. While this setting may not be practically feasible, we need $c > \norm{ \xtrue }_2^2$ to ensure that $X^\natural$ is in the constraint set $\mathcal{X}$. 

The following theorem shows that choosing $c = ( 1 / n ) \sum_{i = 1}^n y_i$ suffices, if the sampling scheme satisfies an isometry property with high probability. 

\begin{prop}
Let $A \in \mathbb{C}^{n \times p}$, whose $i$-th row is given by $a_i^H$. Assume that there exists some $\varepsilon > 0$ such that 
\begin{equation}
( 1 - \varepsilon ) \norm{ \xtrue }_2^2 \leq \norm{ \frac{1}{\sqrt{n}} A \xtrue }_2^2 \leq ( 1 + \varepsilon ) \norm{ \xtrue }_2^2 \label{eq_isometry}
\end{equation}
with probability at least $1 - p_{\epsilon}$. Then we have, for any $t > 0$,
\begin{equation}
\bar{y} := \frac{1}{n} \sum_{i = 1}^n y_i > ( 1 + \varepsilon ) \norm{ \xtrue }_2^2 + t \notag
\end{equation}
with probability at least $1 - p_{\epsilon} - p_t$, where
\begin{equation}
p_t := \exp \left[ - \frac{n t}{4} \log \left( 1 + \frac{t}{2 ( 1 + \varepsilon ) \norm{ \xtrue }_2^2} \right) \right]. \notag
\end{equation}
\end{prop}

If $\xtrue$ is sparse, then the isometry condition (\ref{eq_isometry}) can be implied by the restricted isometry property (RIP) of $A$ \cite{Candes2008,Foucart2013,Rudelson2008}. Even without sparsity, if $n$ is significantly larger than $p$, a matrix $A$ of independent and identically distributed (i.i.d.) subgaussian random variables can also satisfy (\ref{eq_isometry}) with high probability \cite{Foucart2013}.

While the isometry property of the Fourier measurement with a coded diffraction pattern is unclear currently, we show via numerical experiments in Section \ref{sec_experiment} that this rule of thumb works well on both synthetic and real-world data.

\section{Review of Convex Optimization Tools} \label{sec_issues}


We address why several existing convex optimization algorithms are not applicable to (\ref{eq_opt}) in this section.

We note that (\ref{eq_opt}) is a constrained convex minimization problem with a smooth loss function, and there are many well-known algorithms for solving such a problem. State-of-the-art choices for large-scale applications include the proximal gradient-type methods \cite{Auslender2006,Beck2009,Combettes2005,Nesterov2013,Nesterov2013a,Tran-Dinh2015b}, alternating direction method of multipliers (ADMM) \cite{Eckstein2015}, and Frank-Wolfe-type algorithms (a.k.a. conditional gradient methods) \cite{Freund2014,Garber2015,Harchaoui2014,Jaggi2013,Nesterov2015,Yurtsever2015b,Yurtsever2015}. There are also well-developed MATLAB packages available on the Internet \cite{Becker2011,Tran-Dinh2015b}. Those seemingly ready-to-use convex optimization tools, however, are not desirable for solving our problem (\ref{eq_opt}) for two issues. 

The first issue is scalability. When applied to the problem (\ref{eq_opt}), both proximal gradient-type methods and the ADMM require computing the \emph{prox-mapping} given by
\begin{equation}
\mathrm{prox} ( X ) := \arg \min_{S} \set{ \omega ( S - X ) : S \in \mathcal{X} } \notag
\end{equation}
for a given strongly convex ``distance generating function'' (DGF) $\omega$. A standard choice of DGF for matrix variables is $\omega ( X ) := ( 1/2 ) \norm{ X }_F^2$, where $\norm{ \cdot }_F$ denotes the Frobenius norm. For a positive semi-definite matrix $X \in \mathbb{C}^{p \times p}$, whose eigenvalue decomposition is $X = U \mathrm{diag} ( v ) U^H$, we have $\mathrm{prox} ( X ) = U \mathrm{diag} ( \tilde{v} ) U^H$, where $\tilde{v }$ is the Euclidean projection of $v$ onto the standard simplex in $\mathbb{R}^p$ scaled by $c$. While the prox-mapping is simple to describe, the eigenvalue decomposition renders the algorithm slow when the parameter dimension $p$ is large, as its computational complexity is in general $O ( p^3 )$. Similar issues exist when we choose other DGFs.

Scalability is a major reason why Frank-Wolfe-type algorithms have been attracting attention in recent years. We summarize the standard Frank-Wolfe algorithm (when applied to (\ref{eq_opt})) in Algorithm \ref{alg_fw}, where $( \tau_t )_{t = 1}^T$ is a sequence of real numbers in the interval $( 0, 1 ]$ to be specified. 

Here we have a slight abuse of notations. When applied to our specific problem (\ref{eq_opt}), the variables $x_0, \ldots, x_t$ and $\nabla f ( x_t )$ should be understood as their matrix counterparts $X_0, \ldots, X_t$ and $\nabla f ( X_t )$, respectively. 

\begin{algorithm} 
\caption{(The standard Frank-Wolfe algorithm)}
\label{alg_fw}
\begin{algorithmic}
\State Choose an arbitrary $x_0 \in \mathcal{X}$
\For{$t = 0, \ldots, T$}
\State Compute $v_t \in \argmin_{s} \set{ \left\langle s, \nabla f ( x_t ) \right\rangle:  s \in \mathcal{X} }$
\State Update $x_{t+1} = ( 1 - \tau_t ) x_t + \tau_t v_t$
\EndFor
\end{algorithmic}
\end{algorithm}

The only computational bottleneck is in computing $v_t$ (or its matrix counterpart $V_t$). For the specific constraint set $\mathcal{X}$ given in (\ref{eq_X}) and any positive semi-definite matrix $X_t$, it can be easily verified that $V_t$ is a scaled rank-one approximation of $\nabla f ( X_t )$, and hence can be efficiently computed by the Lanczos method \cite{Jaggi2013}. More precisely, let $u_t \in \mathbb{C}^p$ be an eigenvector of $\nabla f ( X_t )$ associated with the largest eigenvalue. We have $V_t = c ( u_t u_t^H )$.

Unfortunately, the second issue arises: none of the existing theoretical convergence guarantees for Frank-Wolfe-type algorithms, to the best of our knowledge, is valid for the specific loss function (\ref{eq_f}). The result in \cite{Jaggi2013} requires a bounded curvature condition; \cite{Freund2014,Garber2015,Harchaoui2014} require the gradient of the objective function to be Lipschitz continuous; \cite{Nesterov2015} requires a weaker condition that the gradient is H\"{o}lder continuous; the Frank-Wolfe like algorithm in \cite{Yurtsever2015b,Yurtsever2015} requires the gradient of the  conjugate of the objective function to be H\"{o}lder continuous. 
All of the conditions mentioned above implicitly presumes $\mathcal{X} \subseteq \mathrm{dom} ( f )$, but this is not the case for \eqref{eq_opt}, since $0 \in \mathcal{X}$ but $0 \notin \mathrm{dom} ( f )$.


The second issue also exists for proximal gradient-type methods and the ADMM, as \cite{Auslender2006,Beck2009,Combettes2005,Eckstein2015,Nesterov2013,Nesterov2013a} also require the Lipschitz continuity of the gradient. The only exception is the composite self-concordant minimization algorithms proposed in \cite{Tran-Dinh2015b}---the logarithmic function is a typical example of self-concordant functions.

There are some works on noiseless phase retrieval by non-convex optimization techniques \cite{Candes2015,Netrapalli2013}, and provide theoretical convergence guarantees. The convergence guarantees do not extend to the Poisson noise case.

\section{Convergence Guarantee}

In this section, we provide convergence guarantee of the standard Frank-Wolfe method in Algorithm \ref{alg_fw} for the prototype constrained convex optimization optimization problem: 
\begin{equation}
g^\star := \min_{X \in \mathcal{C}} \set{ g ( X ) : X \in \mathcal{C} } \label{eq_prototype}
\end{equation}
where $\mathcal{C}$ is a nuclear norm ball in $\mathbb{R}^{p \times p}$, and 
\begin{equation}
g ( X ) := \mathrm{Tr} ( \Psi X ) - \sum_{i = 1}^n \eta_i \log \mathrm{Tr} ( \Phi_i X ) \label{eq_g} 
\end{equation}
for some $\Psi \in \mathbb{R}^{p \times p}$, non-negative integers $\eta_1, \ldots, \eta_n$, and positive semi-definite matrices $\Phi_1, \ldots, \Phi_n \in \mathbb{R}^p$.


We start with some definitions. Let $\norm{ \cdot }$ be the spectral norm on $\mathbb{R}^{p \times p}$, and $\norm{\cdot}_*$ be the nuclear norm.
Define $d_{\mathcal{C}}$ as the diameter of $\mathcal{C}$, i.e., 
\begin{equation}
d_{\mathcal{C}} := \max_{ X, Y } \set{ \norm{ X - Y } : X , Y \in \mathcal{C} }. \notag
\end{equation}
Let $\dA := \max_i \norm{\Phi_i}$ and $\dB:= \norm{ \Psi }$. Furthermore, we define
\begin{align}
\lmax &:= \max_{i,x} \set{ \mathrm{Tr} ( \Phi_i X ) : 1 \leq i \leq n, X \in \mathcal{C} } \notag \\
\lmin &:= \min_i \set{ \mathrm{Tr} ( \Phi_i X_0 ) : 1 \leq i \leq n }. \notag
\end{align}
Notice that we need to choose $X_0$ such that $\lmin > 0$, due to the presence of logarithmic functions in $g$.

Our main theoretical result is the following theorem:

\begin{thm} \label{thm_main}
Consider the optimization problem (\ref{eq_prototype}). The iterates $( X_t )_{t \geq 0}$ given by Algorithm \ref{alg_fw} with 
\begin{equation}
\tau_t := \frac{2}{t + 3} \notag
\end{equation}
satisfies
\begin{equation}
g ( X_t ) - g^\star < \frac{8 \gamma^2 \dA^2 \dX^2 }{t + 2} + \frac{ 2 \dX \norm{ \naf ( X_0 ) } }{ \lmin ( t + 1 ) ( t + 2 ) } \notag
\end{equation}
The quantity $\gamma := \max \set{ \gamma_1, \gamma_2, \gamma_3 }$ is a constant independent of $t$, where
\begin{align}
\gamma_1 & := \frac{ 2 \dB \dX }{ \lmin }, \quad \gamma_2 := 2 \frac{n \dy}{\lmin} \left( \frac{ 4 n \lmax \dy }{\lmin} + 1 \right)^2, \notag \\
\gamma_3 & := \frac{ 64 n^2 \lmax^2 \dy^2 }{ \lmin^3 } \left( \frac{ 4 n \lmax \dy }{ \lmin } + 1 \right). \notag
\end{align}

Consequently, we have $g ( X_t ) - g^\star = O ( 1 / t )$.
\end{thm}

Theorem \ref{thm_main} establishes the validity of using the standard Frank-Wolfe algorithm to solve (\ref{eq_opt}). 
We note that this theorem is a \emph{worst case} guarantee for all loss functions of the form \eqref{eq_g}.
As we will see in the next section, empirically, both the constant and the convergence rate can be much better.

Our choice of $\tau_t$ is slightly different from the standard one in \cite{Jaggi2013,Nesterov2015}, where $\tau_t := 2 / ( t + 2 )$. 
This is due of technical concerns in the proof.

As a short sketch, the key idea is to show the boundedness of $\norm{ \nabla g ( X_{t+1} ) - \nabla g ( X_t ) }$ for all $t$, where $\norm{ \cdot }$ denotes the spectral norm. 
This bound, by the framework in \cite{Nesterov2015}, is sufficient to establish the convergence guarantee. 
This is simple if the gradient is H\"{o}lder continuous, since then
\begin{equation}
\norm{ \nabla g ( X_{t+1} ) - \nabla g ( X_t ) } \leq L_{\nu} \norm{ X_{t+1} - X_t }_*^\nu \leq L_{\nu} \dX^\nu \notag
\end{equation}
for some $\nu \in ( 0, 1]$ and $L_\nu > 0$. 
For the optimization problem (\ref{eq_opt}) we consider, this issue can be reduced to the boundedness of $$C_t := \sum_{i = 1}^n \frac{\eta_i}{\mathrm{Tr} ( \Phi_i X_t )}$$ for all $t$. 
We complete the proof by showing that $C_t$ is bounded above by a constant for all $t$, if we choose $\tau_t = 2 / ( t + 3 )$.


\section{Numerical Results} \label{sec_experiment}
In this section, we present numerical evidence to assess the convergence behaviour and the scalability of the proposed Frank-Wolfe algorithm. 

 \begin{figure*}[!ht]
 \begin{center}
\includegraphics[width=\textwidth]{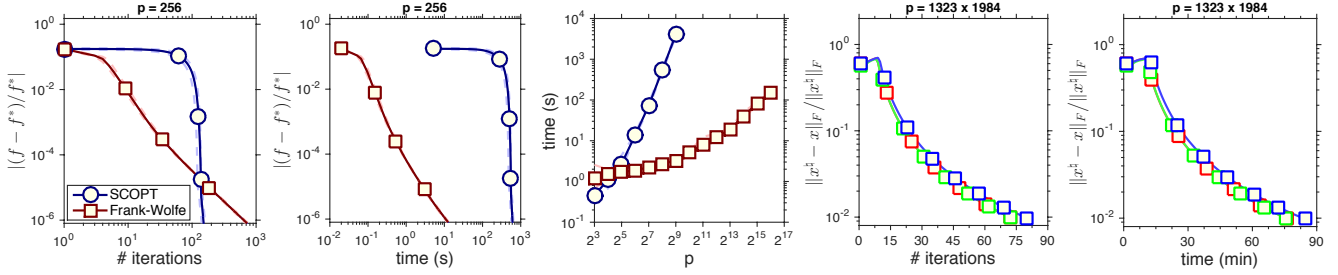} 
\caption{Convergence behaviour of the algorithms for different data sizes: The three plots on the left correspond to the first experiment. Solid lines show the average performance over $10$ random trials, and the two dashed lines show the best and the worst performances, respectively. The two plots on the right correspond to the second experiment. Each color (blue, green, red) represents one color channel.} \label{fig:fig1}
\end{center}
\end{figure*} 

Our numerical experiment is based on coded diffraction pattern measurements with the octonary modulation, which were considered in \cite{ Candes2015b, Yurtsever2015b} for the noiseless model. 
A similar setup was also considered also in \cite{Candes2015a} for the Poisson noise model. 

In \cite{Candes2015a}, the MATLAB package TFOCS \cite{Becker2011} was used to solve a convex optimization problem similar to (\ref{eq_opt}). 
The algorithm, however, is not guaranteed to converge for the problem under our consideration (cf. Section \ref{sec_issues}). 
Therefore, we compare the Frank-Wolfe algorithm with the proximal gradient method in the Self-Concordant OPTimization toolbox (SCOPT) \cite{Tran-Dinh2015b}. Recall that our loss function is self-concordant, and hence the algorithms in \cite{Tran-Dinh2015b} are applicable.


In our first experiment, we consider the random Gaussian signal model: 
We generate a random complex Gaussian vector $x^\natural \in \mathbb{C}^p$ with i.i.d. entries, where the real and the imaginary parts of the each entry of $x^\natural$ are independent and sampled from the standard Gaussian distribution. 

We run both algorithms starting from the same Gaussian initial iterate, sampled from the same distribution as ${x}^\natural$.
We keep track of the objective value and the elapsed time over the iterations, and compute the approximate relative objective residual ($|f - f^\ast | / |f^\ast |$) as the performance measure, where the actual optimum value $f^\star$ is approximated by $f^\ast$, the minimum objective value obtained by running $200$ iterations of the SCOPT and/or $10000$ iterations of the Frank-Wolfe algorithm.


In the second experiment, we test the scalability of the Frank-Wolfe approach, by recovering a real image as in \cite{ Candes2015b, Yurtsever2015b}. 
We choose the EPFL campus image of size $1323 \times 1984$ as the signal to be measured, which corresponds to a signal dimension $p = 2624832$.
We apply the Frank-Wolfe algorithm to recover three color channels separately, and stop the algorithm when $10^{-2}$ recovery error ($\| x - x^\natural\|_F / \| x^\natural \|_F$) is reached. 

In both experiments, we set the constraint parameter $c$ to the mean of the measurements, following the rule of thumb in Section \ref{sec_constraint}, and we set the number of different modulating waveforms $L$ to $20$. 

 \begin{figure}[h]
 \begin{center}
\includegraphics[width=.95\linewidth]{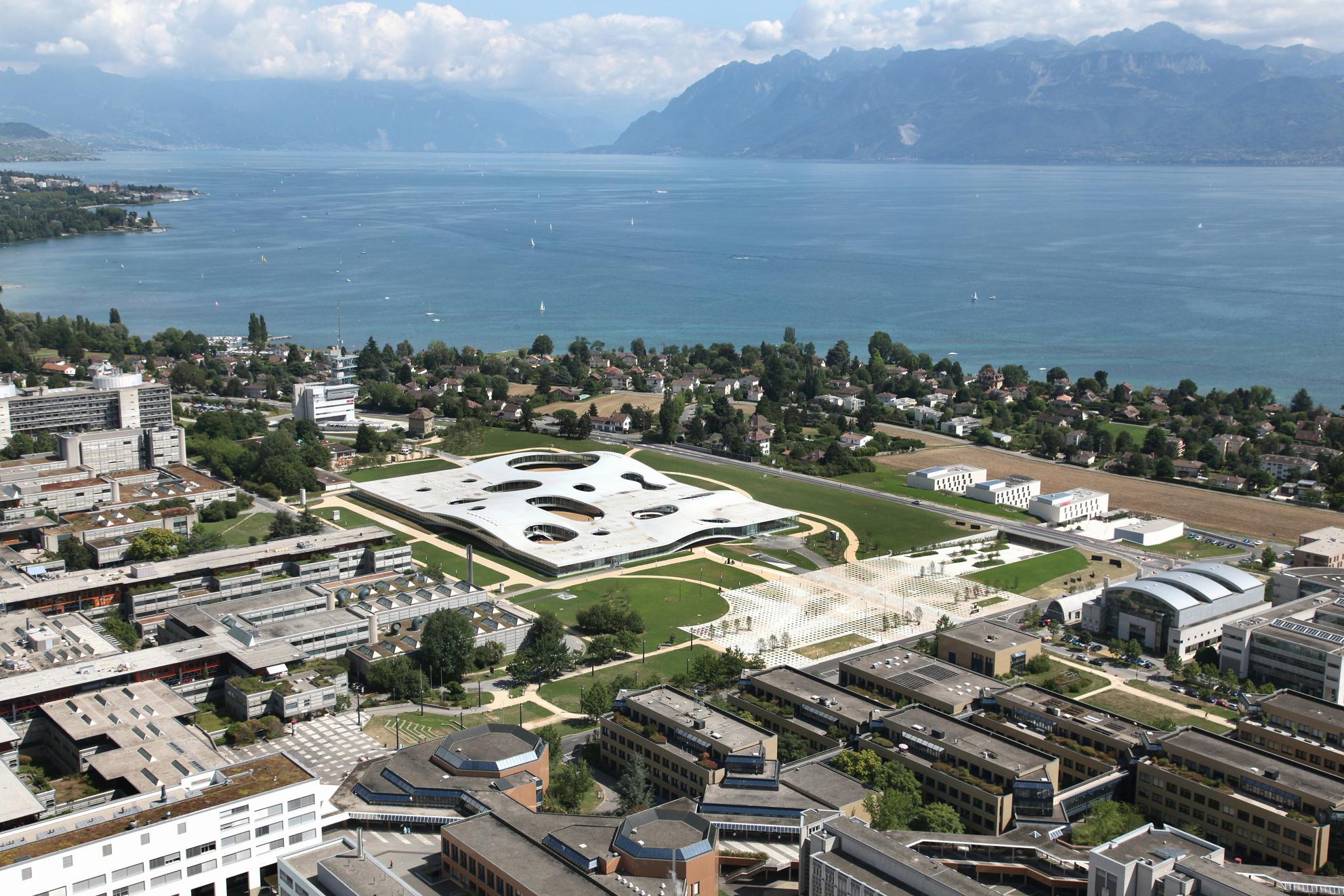} 
\caption{An EPFL image of size $1323 \times 1984$, reconstructed by 75 iterations of the Frank-Wolfe algorithm: PSNR = 44.92 dB.} \label{fig:fig2}
\end{center}
\end{figure} 
We implement Algorithm \ref{alg_fw} in MATLAB and use the built-in \texttt{eigs} function, which is based on the Lanczos algorithm, with $10^{-3}$ relative error tolerance, to perform the minimization step of the Frank-Wolfe algorithm. 
In the weighting step, we adapt the efficient thin singular value decomposition updating method of \cite{Brand2006} under low rank modifications, as explained in \cite{Yurtsever2015b}, in order to tame the memory growth. 

We time our experiments on a computer cluster, and restricting the computational resource to 8 CPU of 2.40 GHz and 32 GB of memory space per simulation. 

Figure \ref{fig:fig1} illustrates the convergence behaviour of the algorithms for different data sizes. 

The first three plots on the left correspond to the first experiment. 
Solid lines show the average performance over $10$ random trials, and the two dashed lines show the best and the worst instances, respectively. 
In the first two plots, we observe that the 
empirical rate of convergence is about $O(t^{-1.89})$, which is better than the theoretically guaranteed rate $O(t^{-1})$. 
In the third plot, we show the time required to reach a predefined accuracy level of $10^{-5}$ in terms of the relative objective residual, for different data sizes.

The last two plots of Figure \ref{fig:fig1} correspond to the second experiment, which also provides an empirical evidence for the estimation quality using the constraint parameter $c$. Each color (blue, green, red) represents one color channel.

Finally, Figure \ref{fig:fig2} shows the estimate $x_t$, after 75 iterations of the Frank-Wolfe method. The PSNR of the reconstructed image is 44.92dB.

Notice that, considering the lifted dimensions $p^2$ in the second experiment, even the generation of a simple iterate $X_t$ would require approximately $7$ TB of memory space, for a single color channel, when using the prox-mapping-based solver in SCOPT. 
By avoiding the computation of the prox-mapping, and adapting the efficient low rank updates, the Frank-Wolfe algorithm keeps a low memory footprint, and hence is more scalable compared to the self-concordant optimization method in SCOPT.

%


\section{Discussion}

While we focus on the Poisson phase retrieval problem in this paper, our main contribution is in verifying the validity of applying the standard Frank-Wolfe algorithm to optimization problems of the form (\ref{eq_opt}). Therefore, the application of our result is not restricted to Poisson phase retrieval. One interesting application is ML estimation for quantum state tomography \cite{Hradil1997}, where the parameter dimension grows exponentially fast with the number of qubits, and the physical model naturally imposes a nuclear norm constraint.

\section{Proofs}

\subsection{Proof of Proposition 3.1}

Notice that, conditioning on $a_1, \ldots, a_n$, $n \bar{y}$ is a Poisson random variable with mean $\sum_{i = 1}^n \lambda_i$.
By the tail bound for Poisson random variables \cite{Bobkov1998,Kontoyiannis2006}, conditioning on $a_1, \ldots, a_n$, we have for any $t > 0$, 
\begin{equation}
\mathbb{P} \set{ \bar{y} - \mathbb{E}\, \bar{y} > t } \leq \exp \left[ - \frac{ n t }{ 4 } \log \left( 1 + \frac{ t }{ 2 \lambda } \right) \right], \notag
\end{equation}
where $\lambda := ( 1 / n ) \sum_{i = 1}^n \lambda_i$.

Recall that $\lambda_i := \abs{ \iprod{ a_i }{ \xtrue } }^2$. 
By the assumption on $A$, we have $( 1 - \delta ) \norm{ x }_2^2 \leq \lambda \leq ( 1 + \delta ) \norm{ x }_2^2$ with probability at least $1 - p_{\delta}$.
Moreover, on this event, we have
\begin{align}
& \mathbb{P} \set{ \bar{y} - ( 1 + \delta ) \norm{ \xtrue }_2^2 > t } \notag \\
& \quad \leq \mathbb{P} \set{ \bar{y} - \lambda > t } \notag \\
& \quad \leq \exp \left[ - \frac{ n t }{ 4 } \log \left( 1 + \frac{ t }{ 2 ( 1 + \delta ) \norm{ \xtrue }_2^2 } \right) \right]. \notag
\end{align}
This proves the theorem.

\subsection{Proof of Theorem 5.1}


Let $( \alpha_t )_{t \geq 0}$, $\alpha_0 \neq 0$ be a sequence of non-negative real numbers. 
We consider step sizes of the form 
\begin{equation}
\tau_t = \alpha_{t+1} / S_{t+1}, \label{eq_general_tau}
\end{equation}
where $S_{t} := \sum_{k = 0}^t \alpha_t$. 
Unless otherwise stated, $( X_t )_{t \geq 0}$ refers to the sequence of iterates generated by Algorithm \ref{alg_fw}, with the step size chosen as in \eqref{eq_general_tau}. Notice that then the convergence rate of the algorithm can depend on the sequence $( \alpha_t )_{t \geq 0}$.

By the convexity of $\mathcal{C}$, it is obvious that $X_t \in \mathcal{C}$ for all $t$. 
Due to the presence of the logarithmic function, we also need to verify that $X_t \in \mathrm{dom} ( g )$ for all $t$.

\begin{prop} \label{PROP_FEASIBILITY}
The following hold.
\begin{enumerate}
\item $\mathrm{Tr} ( \Phi_i V_t ) \geq 0$ for all $i$ and $t$.
\item If $\mathrm{Tr} ( \Phi_i X_0 ) > 0$, then $\mathrm{Tr} ( \Phi_i X_t ) > 0$ for all $i$ and $t$.
\end{enumerate}
\end{prop}

\begin{proof}
See Section \ref{sec_feasibility}.
\end{proof}

Now we show the boundedness of $C_t$ for all $t$, as stated in Section 5. 
Recall that $C_t := \sum_{i = 1}^n ( \eta_i / \mathrm{Tr} ( \Phi_i X_t ) )$.

\begin{lem} \label{LEM_KEY}
For any $T$ such that $1 - 4 n ( \lmax / \lmin ) \dy \tau_T > 0$, we have $C_t \leq C$, where $C$ is a constant independent of $t$ defined as
\begin{equation}
C := \max \set{ \frac{ 2 \dB \dX }{ \lmin } , C_0 \prod_{i = 0}^T \frac{1}{ 1 - \tau_i }, \frac{ 64 n^2 \lmax^2 \dy^2 }{ \lmin^3 \left( 1 - \frac{ 4 n \lmax \dy }{ \lmin } \tau_T \right) } }. \notag
\end{equation}
\end{lem}

\begin{proof}
See Section \ref{sec_lem_key}.
\end{proof}

The following lemma mimics \cite[Lemma 2]{Nesterov2015}.
Define
\begin{align}
B_t & := \alpha_0 \max \set{ \iprod{ \naf ( X_0 ) }{ X_0 - X } : X \in \mathcal{X} } \notag \\
& \quad \, \, + \left( \sum_{k = 1}^t \frac{ \alpha_k^2 }{ S_{k - 1} } \right) \gamma , \notag
\end{align}
where $\gamma := C^2 \dA^2 \dX^2$.

\begin{lem} \label{LEM_NESTEROV}
For any $t \geq 0$ and $X \in \mathcal{X}$, we have
\begin{equation}
S_t g ( X_t ) \leq \sum_{k = 0}^t \left\{ \alpha_k \left[ g ( X_k ) + \iprod{ \naf ( X_k ) }{ X - X_k } \right] \right\} + B_t \notag
\end{equation}
\end{lem}

\begin{proof}
See Section \ref{sec_nesterov}.
\end{proof}

Set $X = X^\star$, a minimizer, in Lemma \ref{LEM_NESTEROV}, and notice that $$g ( X_k ) + \iprod{ \naf ( X_k ) }{ X^\star - X_k } \leq g^\star$$ for all $k$.
We immediately obtain a convergence guarantee for any $( a_t )_{t \geq 0}$.

\begin{cor} 
We have $g ( X_t ) - g^\star \leq ( B_t / S_t )$.
\end{cor}

Now we consider the special case where $\alpha _t = t + 1$. 
As then $S_t = ( t + 1 ) ( t + 2 ) / 2$, this choice corresponds to $\tau_t = 2 / ( t + 3 )$ as in Theorem \ref{thm_main}.

\begin{prop} \label{PROP_FINAL}
Choose $\alpha_t = t + 1$. 
We have 
\begin{equation}
\frac{B_t}{S_t} < \frac{ 8 \left( \max \set{ \gamma_1, \gamma_2, \gamma_3 } \right)^2 \dA^2 \dX^2 }{ t + 2 } + \frac{ 2 \dX \norm{ \naf ( X_0 ) } }{ ( t + 1 ) ( t + 2 ) }, \notag
\end{equation}
where 
\begin{align}
\gamma_1 & := \frac{ 2 \dB \dX }{ \lmin }, \quad \gamma_2 := 2 \frac{n \dy}{\lmin} \left( \frac{ 4 n \lmax \dy }{\lmin} + 1 \right)^2, \notag \\
\gamma_3 & := \frac{ 64 n^2 \lmax^2 \dy^2 }{ \lmin^3 } \left( \frac{ 4 n \lmax \dy }{ \lmin } + 1 \right). \notag
\end{align}
\end{prop}

\begin{proof}
See Section \ref{sec_final}.
\end{proof}

\subsection{Proof of Proposition \ref{PROP_FEASIBILITY}} \label{sec_feasibility}

Recall that $V_t$ is always a positive semi-definite matrix of rank $1$, as discussed in Section 4. 
Since $\Phi_i$ is also positive semi-definite, this implies $\mathrm{Tr} ( \Phi_i  V_t ) \geq 0$ for all $i$ and $t$.

We prove the second claim by induction. 
The second claim holds true for $t = 0$ by assumption.
Suppose $\mathrm{Tr} ( \Phi_i X_t ) > 0$ for some $t \geq 0$ for all $i$. 
Because of the assumption that $\alpha_0 \neq 0$, we always have $\tau_t < 1$ for all $t$. 
Then 
\begin{align}
\mathrm{Tr} ( \Phi_i X_{t+1} ) & = ( 1 - \tau ) \mathrm{Tr} ( \Phi_i X_t ) + \tau_t \mathrm{Tr} ( \Phi_i V_t ) \notag \\
& \geq ( 1 - \tau ) \mathrm{Tr} ( \Phi_i X_t ) > 0, \notag
\end{align}
where the first inequality is by the first claim.

\subsection{Proof of Lemma \ref{LEM_KEY}} \label{sec_lem_key}

Consider the sequence $( C_t )_{t \geq 0}$. 
Roughly speaking, the idea behind the proof is to show that there exists some $T > 0$, such that $C_{t + 1} \leq C_t$ for all $t \geq T$; then we can bound $C_t$ from above by $C_T$ for all $t \geq T$, a constant independent of $t$. 
Notice that, however, the actual argument in this proof is slightly more delicate (cf. the proof of Proposition \ref{prop_key_2}).

A simple bound on $C_{t+1}$ is 
\begin{align}
C_{t+1} &= \sum_{i = 1}^{n} \frac{ \eta_i }{ \mathrm{Tr} ( \Phi_i X_{t+1} ) } \notag \\
&\leq \frac{1}{ ( 1 - \tau_t ) } \sum_{ i = 1 }^n \frac{ \eta_i }{ \mathrm{Tr} ( \Phi_i X_t ) } = \frac{1}{1 - \tau_t} C_t, \label{eq_simple_bound}
\end{align}
using the fact that $\mathrm{Tr} ( \Phi_i X_{t+1} ) \geq ( 1 - \tau_t ) \mathrm{Tr} ( \Phi_i X_t )$. This yields the following simple result.

\begin{prop} \label{prop_simple_bound}
We have $C_t \leq C_0 \prod_{i = 0}^{t} ( 1 - \tau_t )^{-1}$.
\end{prop}

However, as $1 - \tau_t < 1$, the upper bound \eqref{eq_simple_bound} is not sharp enough for our purpose.

Notice that for any $k$, we have
\begin{align}
C_{t+1} & = \sum_{i \neq k} \frac{\eta_i}{\mathrm{Tr} ( \Phi_i X_{t+1} )} + \frac{ \eta_k }{ \mathrm{Tr} ( \Phi_k X_{t+1} ) } \notag \\
& \leq \sum_{i \neq k} \frac{ \eta_i }{ ( 1 - \tau_t ) \mathrm{Tr} ( \Phi_i X_t ) } + \frac{ \eta_k }{ \mathrm{Tr} ( \Phi_k X_{t+1} ) } \notag \\
& = \frac{ C_t }{1 - \tau_t} - \frac{ \eta_k }{ ( 1 - \tau_t ) \mathrm{Tr} ( \Phi_k X_t ) } + \frac{\eta_k}{ \mathrm{Tr} ( \Phi_k X_{t+1} ) } \notag \\
& = \frac{ C_t }{ 1 - \tau_t } - \frac{ \eta_k \tau_t \mathrm{Tr} ( \Phi_k V_t ) }{ \left[ ( 1 - \tau_t ) \mathrm{Tr} ( \Phi_k X_t ) \right] \mathrm{Tr} ( \Phi_k X_{t+1} ) } \notag \\
& \leq  \frac{ C_t }{ 1 - \tau_t } - \xi_k \label{eq_better_bound}
\end{align}
where 
\begin{align}
\xi_k & := \frac{ \tau_t \mathrm{Tr} ( \Phi_k V_t ) }{ \left[ ( 1 - \tau_t ) \mathrm{Tr} ( \Phi_k X_t ) \right] \mathrm{Tr} ( \Phi_k X_{t+1} ) }; \notag
\end{align}
the last inequality is due to the fact that either $\eta_k = 0$ or $\eta_k \geq 1$ in the Poisson phase retrieval problem.
This bound is sharper than \eqref{eq_simple_bound}, as $\xi_k$ is always non-negative.

\begin{prop} \label{prop_key}
If $C_t > 2 \lmin^{-1} \dB \dX$, then there exists some $k \leq n$  such that
\begin{align}
\frac{1}{ \mathrm{Tr} ( \Phi_k X_t ) } &\geq \frac{\lmin C_t}{4 n \lmax \dy }, \notag \\
\mathrm{Tr} ( \Phi_k V_t ) &\geq \frac{\lmin}{4}. \notag
\end{align}
\end{prop}

\begin{proof}

We prove by contradiction. 
By the definition of $V_t$, we have $\iprod{ V_t }{ \naf ( X_t ) } \leq \iprod{ X_0 }{ \naf X_t }$, and hence
\begin{align}
\sum_{i = 1}^n \frac{ \eta_i \iprod{ V_t }{ \Phi_i } }{ \iprod{ X_t }{ \Phi_i } } 
&  \geq \sum_{i = 1}^n \left( \frac{ \eta_i \iprod{ X_0 }{ \Phi_i } }{ \iprod{ X_t }{ \Phi_i } } \right) + \iprod{ \Psi }{ X_0 - V_t } \notag \\
&  \geq \sum_{i = 1}^n \frac{ \eta_i \iprod{ X_0 }{ \Phi_i } }{ \iprod{ X_t }{ \Phi_i } } - \dB \dX \notag \\
&  \geq \lmin C_t - \dB \dX \geq \frac{\lmin C_t}{2}. \notag
\end{align}

Let $\Omega$ be the set of $i$'s such that $\iprod{X_t}{\Phi_i}^{-1} \geq \lmin C_t / ( 4 n \lmax \dy )$. 
Suppose the claim of the proposition is false, i.e. for all $i \in \Omega$, $\iprod{ V_t }{ \Phi_i } < \lmin / 4$.
Then we have
\begin{align}
\sum_{i = 1}^n \frac{ \eta_i \iprod{ V_t }{ \Phi_i } }{ \iprod{ X_t }{ \Phi_i } } 
& = \sum_{i \in \Omega} \frac{ \eta_i \iprod{ V_t }{ \Phi_i } }{ \iprod{ X_t }{ \Phi_i } } + \sum_{i \notin \Omega} \frac{ \eta_i \iprod{ V_t }{ \Phi_i } }{ \iprod{ X_t }{ \Phi_i } } \notag \\
& < \frac{\lmin}{4} C_t + n \dy \lmax \frac{\lmin C_t}{4 n \lmax \dy} = \frac{\lmin C_t}{2}, \notag
\end{align}
a contradiction. This completes the proof.

\end{proof}


Assume $C_t > 2 \lmin^{-1} \dB \dX$. 
By Proposition \ref{prop_key} and \eqref{eq_better_bound}, we have
\begin{align}
& C_{t+1} \notag \\
& \quad \leq C_t \left\{ \frac{1}{1 - \tau_t} - \frac{ \frac{ \tau_t \lmin}{ 4 } }{ ( 1 - \tau_t ) \frac{ 4 n \lmax \dy }{ \lmin } \left[ ( 1 - \tau_t ) \frac{ 4 n \lmax \dy }{ \lmin C_t } + \tau_t \frac{ \lmin }{ 4 } \right] } \right\}. \notag
\end{align}
By direct calculation, we obtain $C_{t+1} \leq C_t$, if
\begin{gather}
1 - \frac{4 n \lmax \dy }{ \lmin } \tau_t > 0, \label{eq_taut_bound} \\
C_t \geq \kappa_t := \frac{ 64 ( 1 - \tau_t ) n^2 \lmax^2 \dy^2 }{ \lmin^3 \left( 1 - \frac{ 4 n \lmax \dy }{ \lmin } \tau_t \right) }. \notag
\end{gather}

\begin{prop} \label{prop_key_2}
Assume that $C_t > 2 \lmin^{-1} \dB \dX$.
Choose $T$ such that \eqref{eq_taut_bound} holds for $t = T$. 
Then we have
\begin{equation}
C_t \leq \max \set{ C_0 \prod_{i = 0}^T \frac{1}{ 1 - \tau_i }, \frac{ 64 n^2 \lmax^2 \dy^2 }{ \lmin^3 \left( 1 - \frac{ 4 n \lmax \dy }{ \lmin } \tau_T \right) } }. \notag
\end{equation}
\end{prop}

\begin{proof}
Since $( \tau_t )_{t \geq 0}$ is a decreasing sequence, the inequality \eqref{eq_taut_bound} holds for all $t \geq T$.

If $t \leq T$, we can apply Proposition \ref{prop_simple_bound}, and obtain
\begin{equation}
C_t \leq C_0 \prod_{i = 0}^t \frac{1}{1 - \tau_i} \leq C_0 \prod_{i = 0}^T \frac{1}{1 - \tau_i}. \notag
\end{equation}

Consider the case when $t > T$. 
Suppose $C_T \geq \kappa_T$. 
We have $C_{t+1} \leq C_t \leq C_T$, which can be bounded using Proposition \ref{prop_simple_bound}, until some $t^*$ such that $C_{t^*} < \kappa_{t^*}$.
But then $C_{t+1} \leq ( 1 - \tau_t )^{-1} \kappa_t$ for all $t \geq t^*$.
If $C_T < \kappa_T$, similarly, we also obtain $C_{t+1} \leq ( 1 - \tau_t )^{-1} \kappa_t$ for all $t \geq T$.
The proposition follows, as 
\begin{equation}
\frac{1}{1 - \tau_t}\kappa_t = \frac{ 64 n^2 \lmax^2 \dy^2 }{ \lmin^3 \left( 1 - \frac{ 4 n \lmax \dy }{ \lmin } \tau_t \right) } \leq \frac{ 64 n^2 \lmax^2 \dy^2 }{ \lmin^3 \left( 1 - \frac{ 4 n \lmax \dy }{ \lmin } \tau_T \right) }. \notag
\end{equation}
\end{proof}

If $C_t \leq 2 \lmin^{-1} \dB \dX$, then this is already a constant upper bound on $C_t$. This completes the proof.

\subsection{Proof of Lemma \ref{LEM_NESTEROV}} \label{sec_nesterov}

We prove by induction.
The claim is obviously correct for $t = 0$.
Suppose the claim holds for some $t \geq 0$.
Then we have
\begin{align}
& \sum_{k = 0}^{t + 1} \alpha_k \left[ g ( X_k ) + \iprod{ \naf ( X_k ) }{ X - X_k } \right] + B_t \notag \\
& \quad \geq S_t g ( X_t ) + \alpha_{k+1} \left[ g ( X_{t+1} ) + \iprod{ \naf ( X_{t+1} ) }{ X - X_{t+1} } \right] \notag \\
& \quad = S_{t+1} g ( X_{t+1} ) + S_t \left[ g ( X_t ) - g ( X_{t+1} ) \right] \notag \\
& \quad \quad \, \, + \iprod{ \naf ( X_{t+1} ) }{ \alpha_{t+1} ( X - X_{t+1} ) }. \notag \\
& \quad \geq S_{t+1} g ( X_{t+1} ) \notag \\
& \quad \quad \, \, + \iprod{ \naf ( X_{t+1} ) }{ \alpha_{t+1} ( X - X_{t+1} ) + S_t ( X_t - X_{t+1} ) } \notag \\
& \quad = S_{t+1} g ( X_{t+1} ) + \alpha_{t+1} \iprod{ \naf ( X_{t+1} ) }{ X - V_t } \notag \\
& \quad \geq S_{t+1} g ( X_{t+1} ) + \alpha_{t+1} \iprod{ \naf ( X_{t+1} ) - \naf ( X_t ) }{ X - V_t }, \notag
\end{align}
where the second inequality is due to convexity of $g$, and the third inequality is due to the fact that
\begin{equation}
\iprod{ \naf ( X_t ) }{ X - V_t } \geq 0 \notag
\end{equation}
for any $X \in \mathcal{C}$, as $V_t$ minimizes $\iprod{ \naf ( X_t ) }{ \cdot }$ on $\mathcal{C}$.

To complete the proof, we need to show that
\begin{equation}
\alpha_{t+1} \iprod{ \naf ( X_{t+1} ) - \naf ( X_t ) }{ X - V_t } \geq B_t - B_{t+1} = - \frac{\alpha_{t+1}^2}{S_t} \gamma, \notag
\end{equation}
or 
\begin{equation}
\iprod{ \naf ( X_{t+1} ) - \naf ( X_t ) }{ X - V_t } \geq - \frac{\alpha_{t+1}}{S_t} \gamma .
\end{equation}
By H\"{o}lder's inequality, we have
\begin{align}
& \abs{ \iprod{ \naf ( X_{t+1} ) - \naf ( X_t ) }{ X - V_t } } \notag \\
& \quad \leq \norm{ \naf ( X_{t+1} ) - \naf ( X_t ) } \norm{ X - V_t }_* \notag \\
& \quad \leq \norm{ \naf ( X_{t+1} ) - \naf ( X_t ) } \dX , \notag
\end{align}
where $\norm{ \cdot }$ denotes the spectral norm.

Now we bound the quantity $\norm{ \naf ( X_{t+1} ) - \naf ( X_t ) }$. 
By direct calculation, we obtain
\begin{align}
& \norm{ \naf ( X_{t+1} ) - \naf (X_t) } \notag \\
& \quad = \norm{ \sum_{i = 1}^n \frac{ \eta_i \iprod{ X_t - X_{t+1} }{ \Phi_i } }{ \iprod{ X_{t} }{ \Phi_i } \iprod{ X_{t+1} }{ \Phi_i } } \Phi_i } \notag \\
& \quad \leq \dA \sum_{i = 1}^n \frac{ \eta_i \abs{ \iprod{ X_t - X_{t+1} }{ \Phi_i } } }{ \iprod{ X_{t} }{ \Phi_i } \iprod{ X_{t+1} }{ \Phi_i } } \notag \\
& \quad = \tau_t \dA \sum_{i = 1}^n \frac{ \eta_i \abs{ \iprod{ X_t - V_t }{ \Phi_i } } }{ \iprod{ X_{t} }{ \Phi_i } \iprod{ X_{t+1} }{ \Phi_i } } \notag \\
& \quad \leq \tau_t \dA \sum_{i = 1}^n \frac{ \eta_i \norm{ X_t - V_t }_* \norm{ \Phi_i } }{ \iprod{ X_t }{ \Phi_i } \iprod{ X_{t+1} }{ \Phi_i } } \notag \\
& \quad \leq \tau_t \dA^2 \dX \sum_{i = 1}^n \frac{ \eta_i }{ \iprod{ X_t }{ \Phi_i } \iprod{ X_{t+1} }{ \Phi_i } }. \notag
\end{align}
Since either $\eta_i = 0$ or $\eta_i \geq 1$, we have
\begin{align}
& \norm{ \naf ( X_{t+1} ) - \naf (X_t) } \notag \\
& \quad \leq \tau_t \dA^2 \dX \sum_{i = 1}^n \frac{ \eta_i^2 }{ \iprod{ X_t }{ \Phi_i } \iprod{ X_{t+1} }{ \Phi_i } } \notag \\
& \quad \leq \frac{ \tau_t \dA^2 \dX }{ 1 - \tau_t } \sum_{i = 1}^n \left( \frac{ \eta_i }{ \iprod{ X_t }{ \Phi_i } } \right)^2 \notag \\
& \quad \leq \frac{ \tau_t }{ 1 - \tau_t } \dA^2 \dX \left( \sum_{i = 1}^n \frac{ \eta_i }{ \iprod{ X_t }{ \Phi_i } } \right)^2 \notag \\
& \quad \leq \frac{ \alpha_{t+1} }{ S_t } \dA^2 \dX \left( \sum_{i = 1}^n \frac{ \eta_i }{ \iprod{ X_t }{ \Phi_i } } \right)^2. \notag
\end{align}
By Lemma \ref{LEM_KEY}, 
\begin{equation}
\norm{ \naf ( X_{t+1} ) - \naf (X_t) } \leq \frac{ \alpha_{t+1} }{ S_t } \dA^2 \dX C^2. \notag
\end{equation}
Hence it suffices to choose $\gamma \geq C^2 \dA^2 \dX^2$.

\subsection{Proof of Proposition \ref{PROP_FINAL}} \label{sec_final}

By H\"{o}lder's inequality, the first term in the definition of $B_t$ can be bounded above by $\norm{ \naf ( X_0 ) } \dX$.
The second term can be bounded as
\begin{align}
\left( \sum_{k = 1}^t \frac{ \alpha_k^2 }{ S_{k - 1} } \right) \gamma = \gamma \sum_{k = 1}^t \left( 2 + \frac{2}{k} \right) \leq 4 t \gamma. \notag
\end{align}
Then we obtain
\begin{align}
\frac{B_t}{S_t} & \leq \frac{ 8 t \gamma }{ ( t + 1 ) ( t + 2 ) } + \frac{ 2 \dX \norm{ \naf ( X_0 ) } }{ ( t + 1 ) ( t + 2 ) } \notag \\
& < \frac{ 8 \gamma }{ t + 2 } + \frac{ 2 \dX \norm{ \naf ( X_0 ) } }{ ( t + 1 ) ( t + 2 ) } \notag \\
& \leq \frac{ 8 C^2 \dA^2 \dX^2 }{ t + 2 } + \frac{ 2 \dX \norm{ \naf ( X_0 ) } }{ ( t + 1 ) ( t + 2 ) }. \notag
\end{align}

The definition of $C$ in Lemma \ref{LEM_KEY} also involves $\tau_t$. 
We notice that choosing $T = 8 n ( \lmax / \lmin ) \dy - 1$ suffices to ensure $1 - 4 n ( \lmax / \lmin ) \dy \tau_T \geq 0$.
Then we obtain
\begin{align}
\prod_{k = 0}^T \frac{1}{1 - \tau_k} & = \frac{ ( T + 2 ) ( T + 3 ) }{ 2 } \notag  \\
& < \frac{ ( T + 3 )^2 }{2} = 2 \left( \frac{ 4 n \lmax \dy }{\lmin} + 1 \right)^2. \notag
\end{align}
The quantity $C_0$ can be easily bounded as $C_0 \leq n \lmin^{-1} \dy$.
Finally, we have
\begin{equation}
\frac{ 64 n^2 \lmax^2 \dy^2 }{ \lmin^3 \left( 1 - \frac{ 4 n \lmax \dy }{ \lmin } \tau_T \right) } = \frac{ 64 n^2 \lmax^2 \dy^2 }{ \lmin^3 } \left( \frac{ 4 n \lmax \dy }{ \lmin } + 1 \right). \notag
\end{equation}

\bibliographystyle{IEEEtranS}
\bibliography{list}

\begin{thebibliography}{10}
\providecommand{\url}[1]{#1}
\csname url@samestyle\endcsname
\providecommand{\newblock}{\relax}
\providecommand{\bibinfo}[2]{#2}
\providecommand{\BIBentrySTDinterwordspacing}{\spaceskip=0pt\relax}
\providecommand{\BIBentryALTinterwordstretchfactor}{4}
\providecommand{\BIBentryALTinterwordspacing}{\spaceskip=\fontdimen2\font plus
\BIBentryALTinterwordstretchfactor\fontdimen3\font minus
  \fontdimen4\font\relax}
\providecommand{\BIBforeignlanguage}[2]{{%
\expandafter\ifx\csname l@#1\endcsname\relax
\typeout{** WARNING: IEEEtranS.bst: No hyphenation pattern has been}%
\typeout{** loaded for the language `#1'. Using the pattern for}%
\typeout{** the default language instead.}%
\else
\language=\csname l@#1\endcsname
\fi
#2}}
\providecommand{\BIBdecl}{\relax}
\BIBdecl

\bibitem{Auslender2006}
A.~Auslender and M.~Teboulle, ``Interior gradient and proximal methods for
  convex and conic optimization,'' \emph{SIAM J. Optim.}, vol.~16, no.~3, pp.
  697--725, 2006.

\bibitem{Beck2009}
A.~Beck and M.~Teboulle, ``A fast iterative shrinkage-thresholding algorithm
  for linear inverse problems,'' \emph{SIAM J. Imaging Sci.}, vol.~2, no.~1,
  pp. 183--202, 2009.

\bibitem{Becker2011}
S.~R. Becker, E.~J. Cand\`{e}s, and M.~C. Grant, ``Templates for convex cone
  problems with applications to sparse signal recovery,'' \emph{Math. Prog.
  Comp.}, vol.~3, pp. 165--218, 2011.

\bibitem{Bobkov1998}
S.~G. Bobkov and M.~Ledoux, ``On modified logarithmic {S}obolev inequalities
  for {B}ernoulli and {P}oisson meausres,'' \emph{J. Funct. Anal.}, vol. 156,
  pp. 347--365, 1998.

\bibitem{Brand2006}
M.~Brand, ``Fast low-rank modifications of the thin singular value
  decomposition,'' \emph{Linear Algebra Appl.}, vol. 415, no.~1, pp. 20--30,
  2006.

\bibitem{Candes2008}
E.~J. Cand\`{e}s, ``The restricted isometry property and its implications for
  compressed sensing,'' \emph{C. R. Acad. Sci. Paris, Ser. I}, vol. 346, pp.
  589--592, 2008.

\bibitem{Candes2015}
E.~J. Cand{\`{e}}s, Y.~C. Eldar, T.~Strohmer, and V.~Voroninski, ``Phase
  retrieval via matrix completion,'' \emph{SIAM Rev.}, vol.~57, no.~2, pp.
  225--251, 2015.

\bibitem{Candes2015a}
E.~J. Cand{\`{e}}s, X.~Li, and M.~Soltanolkotabi, ``Phase retrieval from coded
  diffraction patterns,'' \emph{Appl. Comput. Harmon. Anal.}, vol.~39, pp.
  277--299, 2015.

\bibitem{Candes2015b}
------, ``Phase retrieval via {W}irtinger flow: Theory and algorithms,''
  \emph{IEEE Trans. Inf. Theory}, vol.~61, no.~4, pp. 1985--2007, 2015.

\bibitem{Candes2011b}
E.~J. Cand{\`{e}}s and Y.~Plan, ``Tight oracle inequalities for low-rank matrix
  recovery from a minimal number of noisy random measurements,'' \emph{IEEE
  Trans. Inf. Theory}, vol.~57, no.~4, pp. 2342--2359, 2011.

\bibitem{Candes2013}
E.~J. Cand{\`{e}}s, T.~Strohmer, and V.~Voroninski, ``{P}hase{L}ift: Exact and
  stable signal recovery from magnitude measurements via convex programming,''
  \emph{Commun. Pure Appl. Math.}, vol. LXVI, pp. 1241--1274, 2013.

\bibitem{Combettes2005}
P.~L. Combettes and V.~R. Wajs, ``Signal recovery by proximal forward-backward
  splitting,'' \emph{Multiscale Model. Simul.}, vol.~4, no.~4, pp. 1168--1200,
  2005.

\bibitem{Davenport2014}
M.~A. Davenport, Y.~Plan, E.~van~den Berg, and M.~Wootters, ``1-bit matrix
  completion,'' \emph{Inf. Inference}, vol.~3, pp. 189--223, 2014.

\bibitem{Eckstein2015}
J.~Eckstein and W.~Yao, ``Understanding the convergence of the alternating
  direction method of multipliers: Theoretical and computational
  perspectives,'' 2015.

\bibitem{Fienup1982}
J.~R. Fienup, ``Phase retrieval algorithms: a comparison,'' \emph{Appl. Opt.},
  vol.~21, no.~15, pp. 2758--2769, 1982.

\bibitem{Foucart2013}
S.~Foucart and H.~Rauhut, \emph{A mathematical introduction to compressive
  sensing}.\hskip 1em plus 0.5em minus 0.4em\relax Basel: Birkh\"{a}user, 2013.

\bibitem{Freund2014}
R.~M. Freund and P.~Grigas, ``New analysis and results for the {F}rank-{W}olfe
  method,'' \emph{Math. Program., Ser. A}, 2014.

\bibitem{Garber2015}
D.~Garber and E.~Hazan, ``Faster rates for the {F}rank-{W}olfe method over
  strongly-convex sets,'' in \emph{Proc. 32nd Int. Conf. Machine Learning},
  2015.

\bibitem{Harchaoui2014}
Z.~Harchaoui, A.~Juditsky, and A.~Nemirovski, ``Conditional gradient algorithms
  for norm-regularized smooth convex optimization,'' \emph{Math. Program., Ser.
  A}, vol. 152, no.~1, pp. 75--112, 2015.

\bibitem{Hradil1997}
Z.~Hradil, ``Quantum-state estimation,'' \emph{Phys. Rev. A}, vol.~55, no.~3,
  1997.

\bibitem{Jaggi2013}
M.~Jaggi, ``Revisiting {F}rank-{W}olfe: Projection-free sparse convex
  optimization,'' in \emph{Proc. 30th Int. Conf. Machine Learning}, 2013.

\bibitem{Kontoyiannis2006}
I.~Kontoyiannis and M.~Madiman, ``Measure concentration for compound {P}oisson
  distributions,'' \emph{Electron. Commun. Probab.}, vol.~11, pp. 45--57, 2006.

\bibitem{Nesterov2013}
Y.~Nesterov, ``Gradient methods for minimizing composite functions,''
  \emph{Math. Program., Ser. B}, vol. 140, pp. 125--161, 2013.

\bibitem{Nesterov2015}
------, ``Complexity bounds for primal-dual methods minimizing the model of
  objective function,'' Center for Operations Research and Econometrics, {CORE}
  Discussion Paper, 2015.

\bibitem{Nesterov2013a}
Y.~Nesterov and A.~Nemirovski, ``On first-order algorithms for $\ell_1$/nuclear
  norm minimization,'' \emph{Acta Numer.}, pp. 509--575, 2013.

\bibitem{Netrapalli2013}
P.~Netrapalli, P.~Jain, and S.~Sanghavi, ``Phase retrieval using alternating
  minimization,'' in \emph{Adv. Neural Information Processing Systems 26},
  2013.

\bibitem{Recht2010}
B.~Recht, M.~Fazel, and P.~A. Parrilo, ``Guaranteed minimum-rank solutions of
  linear matrix equations via nuclear norm minimization,'' \emph{SIAM Rev.},
  vol.~52, no.~3, pp. 471--501, 2010.

\bibitem{Rudelson2008}
M.~Rudelson and R.~Vershynin, ``On sparse reconstruction from {F}ourier and
  {G}aussian measurements,'' \emph{Commun. Pure Appl. Math.}, vol. LXI, pp.
  1025--1045, 2008.

\bibitem{Shechtman2015}
Y.~Shechtman, Y.~C. Eldar, O.~Cohen, H.~N. Chapman, J.~Miao, and M.~Segev,
  ``Phase retrieval with application to optical imaging,'' \emph{IEEE Sig.
  Process. Mag.}, vol.~32, no.~3, pp. 87--109, 2015.

\bibitem{Tran-Dinh2015b}
Q.~Tran-Dinh, A.~Kyrillidis, and V.~Cevher, ``Composite self-concordant
  minimization,'' \emph{J. Mach. Learn. Res.}, vol.~16, pp. 371--416, 2015.

\bibitem{Yurtsever2015b}
A.~Yurtsever, Y.-P. Hsieh, and V.~Cevher, ``Scalable convex methods for phase
  retrieval,'' in \emph{6th IEEE Int. Workshop Computational Advances in
  Multi-Sensor Adaptive Processing}, 2015.

\bibitem{Yurtsever2015}
A.~Yurtsever, Q.~Tran-Dinh, and V.~Cevher, ``A universal primal-dual convex
  optimization framework,'' in \emph{29th Ann. Conf. Neural Information
  Processing Systems}, 2015.

\end{thebibliography}

\end{document}